\documentclass[12pt]{article}

\usepackage{verbatim}
\usepackage{enumerate}

\usepackage{color, float}

\usepackage{amssymb,amsthm,amsmath}
\usepackage{latexsym}
\usepackage{graphicx}
\usepackage{hyperref}
\usepackage{float}
\usepackage{algorithm}
\usepackage{algpseudocode}  % Changed from algorithmic
\usepackage{url}
\usepackage{float}
\usepackage{enumitem} 
\usepackage{subcaption}
\usepackage{url}
\usepackage{fullpage}

\usepackage{tikz}
\usetikzlibrary{arrows.meta, calc}
\usetikzlibrary{graphs, positioning, shapes.geometric}
\usetikzlibrary{decorations.markings}

\tikzset{
	mid arrow/.style={
		postaction={
			decorate,
			decoration={
				markings,
				mark=at position 0.8 with {
					\arrow{Latex[length=4mm, width=3mm]}
				}
			}
		}
	}
}

\newtheorem{theorem}{Theorem}%[section]

\newtheorem{conj}[theorem]{Conjecture}
\newtheorem{cor}[theorem]{Corollary}
\newtheorem{obs}[theorem]{Observation}
\newtheorem{remark}{Remark}
\newtheorem{defi}[theorem]{Definition}
\newtheorem{prop}[theorem]{Proposition}

\title{On Link-irregular Digraphs} 

\author{
Alexander Bastien
\qquad
Omid Khormali\thanks{Corresponding author. Email: \texttt{ok16@evansville.edu}} \\
\small Department of Mathematics\\[-0.8ex]
\small University of Evansville\\[-0.8ex]
\small Evansville, Indiana 47722, USA.\\
\small \texttt{ab995@evansville.edu}\\
\small \texttt{ok16@evansville.edu}
}

\begin{document}

\maketitle

%%%%%%%%%%%%%%%%%%%%%%%%%%%%%%%%%%%%%%%%%%%%%%%%%%%%%%%%%%%%%%%
\begin{abstract}
\noindent We extend the study of link-irregular graphs to directed graphs (digraphs), where a digraph is link-irregular if no two vertices have isomorphic directed links. We establish that link-irregular digraphs exist on $n$ vertices if and only if $n \geq 5$, and prove that their underlying graphs must contain 3-cycles. We conjecture that link-irregular tournaments exist if and only if $n \geq 6$, providing explicit constructions for $n \leq 8$ and computational verification for $n \leq 100$. We derive lower bounds on the minimum degree and outdegree required for link-irregularity, establish that almost all link-irregular digraphs are nonplanar, and prove that any link-irregular orientable graph admits a link-irregular labeling. Additionally, we construct explicit examples of link-irregular digraphs with constant outdegree and regular tournaments.
\end{abstract}



\section{Introduction}

In this paper, we follow standard terminology and notation in graph theory as presented in West's textbook \cite{west}. For a graph $G$ with vertex set $V(G)$ and edge set $E(G)$, the number of vertices and edges in a graph $G$ are denoted by $n(G)$ and $e(G)$, respectively. The degree of a vertex $u \in V(G)$, $d(u)$, is the number of edges connected to vertex $u$. And $D(G)$ is the set of degrees of the vertices of graph $G$. A graph is irregular if no two vertices in the graph have the same degree. It is known that no such graph exists because there is no simple graph in which all vertex degrees are distinct. Ali, Chartrand, and Zhang~\cite{akbar_book} introduced the notions of the \emph{link} of a vertex, \emph{link-regular} graphs, and \emph{link-irregular} graphs. The \emph{link} $L(v)$ of a vertex $v$ in a graph $G$ is defined as the subgraph induced by the neighborhood of $v$; that is, $L(v) = G[N(v)]$. A graph $G$ is \emph{link-irregular} if every pair of distinct vertices has non-isomorphic links; that is, $L(u) \ncong L(v)$ for all distinct $u, v \in V(G)$. \\

Recently, Bastien and Khormali in \cite{alexander} obtained the results: there are no bipartite or regular link-irregular graphs for small orders, and only finitely many link-irregular graphs are planar. Also, they disprove a conjecture that no regular link-irregular graphs exist, through explicit and probabilistic constructions. \\

Also, Bastien and Khormali in \cite{alexander2} introduced link-irregular labelings for undirected graphs by assigning positive integer labels to edges, effectively modeling loopless multigraphs while working within the framework of simple graphs. This allowed them to distinguish each vertex by the labeled subgraph induced by its neighbors. They characterized the existence of such labelings, defined the link-irregular labeling number, and showed that while families like bipartite graphs and cycles do not admit such labelings, complete and wheel graphs do. They also proved that any labeling number can be realized by some graph. \\

In this paper, we extend the study to directed graphs, examining how directionality impacts link-irregularity. The directed \emph{link} $L_D(v)$ of a vertex $v$ in a digraph $D$ is defined as the subgraph induced by the neighborhood of $v$; that is, $L_D(v) = D[N^+(v)\cup N^-(v)]$ 
where $N^+(v)$ denotes the set of vertices with arcs from $v$ (out-neighbors), and $N^-(v)$ denotes the set of vertices with arcs to $v$ (in-neighbors). %\textcolor{red}{maybe somwhere here, we should note that if the graph $D$ is understood to be a digraph then we write $L(v)$ for the directed link of $D$}. 
Note that when $D$ is understood to be a digraph, we will simply write $L(v)$ for the directed link. A digraph $D$ is \emph{link-irregular} if every pair of distinct vertices has non-isomorphic directed links; that is, $L_D(u) \ncong L_D(v)$ for all distinct $u, v \in V(G)$.\\

In the following, we investigate the structural properties of the labeled link-irregular graphs.

%%%%%%%%%%%%%%%%%%%%%%%%%%%%%%%%%%%%%%%%%%%%%%%%%
\section{Results}

By a straightforward observation, we obtain the following result.

% \begin{obs}
% Let 
% $D$ be a digraph whose underlying undirected graph is one of the following:
% \begin{itemize}
% \item a tree,
% \item a bipartite graph, 
% \item a hypercube, or
% \item a wheel graph.
% \end{itemize}
% Then, $D$ is not a link-irregular digraph.
% \end{obs}
\begin{obs}
    If the underlying undirected graph of a digraph $D$ is a tree, bipartite 
    graph, hypercube, or wheel graph, then $D$ is not link-irregular.
\end{obs}
\begin{proof}
For tree, hypercube and bipartite graphs, since the underlying graphs are all link-empty (for each vertex, the induced subgraph on its neighbors is empty), the corresponding directed links in the digraph are also empty. As a result, all vertices in $D$ have identical (empty) directed link subgraphs, and thus $D$ cannot be link-irregular.\\

For wheels, consider the smallest wheel, $W_3$, which consists of a triangle (3-cycle) and a central vertex connected to each vertex of the cycle. In this case, the underlying links (ignoring direction) of the three cycle vertices are each isomorphic to the path $P_3$. Each pair of these links shares a common edge, and any assignment of directions to the arcs will necessarily result in at least two directed links being isomorphic. Therefore, no orientation of $W_3$ yields a link-irregular digraph.\\

Now consider $W_4$, which contains a 4-cycle and a central hub vertex. Any orientation of $W_4$ results in two non-adjacent cycle vertices will have the same induced directed link (up to isomorphism) due to the graph's symmetry. So the digraph cannot be link-irregular under any orientation.\\

More generally, for $W_n$ with $n \geq 5$, although the number of distinct links increases, the number of possible orientations of links remains limited. Since each edge has only two possible directions, the number of distinct induced directed links is smaller than the number of vertices. Furthermore, the structural symmetry of the wheel graph ensures that multiple vertices will have isomorphic neighborhoods, and hence isomorphic directed links. Therefore, for any orientation of $W_n$, there will exist at least two vertices with isomorphic directed links, and the digraph is not link-irregular.
\end{proof}

Before we move on to our next result, we require the following result from \cite{behzad}.

\begin{prop} \cite{behzad} \label{noperfect}
    If $G$ is a graph, then there exist two vertices $u, v \in V(G)$ such that $d(u) = d(v)$.
\end{prop}
We are now ready to state the following observation.
\begin{obs}
    If a digraph $D$ is link-irregular, then its underlying graph contains a 3-cycle.
\end{obs}
\begin{proof}
    Suppose $D$ is a link-irregular digraph. By Proposition \ref{noperfect}, $D$ has two vertices with the same degree. Call this common degree $d$. Since $D$ is link-irregular, these cannot both have $\overline{K}_d$ as their links, hence $L(u)$ or $L(v)$ has an edge, so this makes a triangle in the underlying graph. This completes the proof. 
\end{proof}

The following result is about the existence of link-irregular digraphs on a given number of vertices.

\begin{theorem} \label{existencethm}
There exists a link-irregular digraph $D$ with $|D| = n$ if and only if $n \geq 5$.
\end{theorem}
\begin{proof}
    If $D$ has $4$ or fewer vertices, we simply check all possible underlying graphs of $D$ to verify that $D$ cannot be link-irregular. If $D$ has two vertices, these vertices both either have the graph with no vertices or $K_1$ as their link. In either case, they both have the same link. If $D$ has three vertices, there are either two vertices of degree $0$, two vertices of degree $1$, or three vertices of degree $2$. In any case, there must be two vertices with the same link. If $D$ has four vertices, we see through observing all possible underlying graphs that each graph has two vertices of degree $0$, two vertices of degree $1$, two vertices of degree $2$, or four vertices of degree $3$. In the first three cases, these two vertices have isomorphic links. In the third case, $D$ is an orientation of $K_4$. Since there are only two distinct orientations of $K_3$ but there are four vertices with link $K_3$, some two vertices must have the same link. Hence $D$ is not link-irregular.\\
    
    Now we show that there exists a link-irregular digraph on $n$ vertices for $n \geq 5$. If $n=5$, 
    the link irregular graphs on 5 vertices are shown in Figure \ref{tournament5}.
    %the following is a link-irregular digraph on $n$ vertices. \textcolor{red}{Insert link-irregular digraph on 5 vertices here}.\\
    If $n \geq 6$, then we use the result from \cite{akbar_paper} that says that there is a link-irregular graph on $n$ vertices for $n \geq 6$. It is evident that any orientation of a link-irregular graph is a link-irregular digraph, so this proves our result.
\end{proof}

In \cite{akbar_paper}, Ali, Chartrand, and Zhang claimed that there exists a unique link-irregular graph on a minimum number of vertices. That is, they stated that there is only one link-irregular graph on $6$ vertices. This might lead the reader to ask uniqueness questions regarding link-irregular digraphs. In particular, the reader might ask whether there is any number $n$ such that there is only one link-irregular digraph on $n$ vertices. It turns out that there is no such number $n$. We state this in our observation.

\begin{obs}
    For any number $n=1,2,3 \cdots$, there is either no link-irregular digraph on $n$ vertices, or there is more than one link-irregular digraph on $n$ vertices.
\end{obs}
\begin{proof}
    The result is clear for $n < 5$. Suppose $n = 5$. Then the two following digraphs (which happen to be non-isomorphic orientations of the same underlying graph) are distinct link-irregular digraphs on $n$ vertices. 

\begin{figure}[H]
		\centering
		\begin{tabular}{c@{\hspace{2cm}}c}
			% --- (a) Original T_6 graph ---
			\begin{tikzpicture}[scale=1,
				midarrow/.style={
					decoration={markings,
						mark=at position 0.45 with {\arrow{Stealth[length=6mm, width=2.5mm]}}
					},
					postaction={decorate}
				}
				]
				
				% --- Vertex positions (matching your drawing) ---
				\node[circle, fill=black, inner sep=2.5pt, label=below:1] (v1) at (1.5,-1.2) {};
				\node[circle, fill=black, inner sep=2.5pt, label=below:2] (v2) at (-1.5,-1.2) {};
				\node[circle, fill=black, inner sep=2.5pt, label=left:3]  (v3) at (-2,0) {};
				\node[circle, fill=black, inner sep=2.5pt, label=above:4] (v4) at (0,1.2) {};
				\node[circle, fill=black, inner sep=2.5pt, label=right:5] (v5) at (2,0) {};
				
				% --- Directed edges (as in your photo) ---
				\draw[very thick, midarrow] (v3) -- (v4);
				\draw[very thick, midarrow] (v4) -- (v5);
				\draw[very thick, midarrow] (v2) -- (v3);
				\draw[very thick, midarrow] (v5) -- (v1);
				\draw[very thick, midarrow] (v5) -- (v3);
				\draw[very thick, midarrow] (v4) -- (v1);
				\draw[very thick, midarrow] (v2) -- (v4);
				\draw[very thick, midarrow] (v2) -- (v5);
				
			\end{tikzpicture}
			&
			% --- (b) Extended graph with T_6 ---
			\begin{tikzpicture}[scale=1,
				midarrow/.style={
					decoration={markings,
						mark=at position 0.45 with {\arrow{Stealth[length=6mm, width=2.5mm]}}
					},
					postaction={decorate}
				}
				]
				
				% --- Vertex positions (matching your drawing) ---
				\node[circle, fill=black, inner sep=2.5pt, label=below:1] (v1) at (1.5,-1.2) {};
				\node[circle, fill=black, inner sep=2.5pt, label=below:2] (v2) at (-1.5,-1.2) {};
				\node[circle, fill=black, inner sep=2.5pt, label=left:3]  (v3) at (-2,0) {};
				\node[circle, fill=black, inner sep=2.5pt, label=above:4] (v4) at (0,1.2) {};
				\node[circle, fill=black, inner sep=2.5pt, label=right:5] (v5) at (2,0) {};
				
				% --- Directed edges (as in your photo) ---
				\draw[very thick, midarrow] (v3) -- (v4);
				\draw[very thick, midarrow] (v4) -- (v5);
				\draw[very thick, midarrow] (v2) -- (v3);
				\draw[very thick, midarrow] (v1) -- (v5);
				\draw[very thick, midarrow] (v5) -- (v3);
				\draw[very thick, midarrow] (v4) -- (v1);
				\draw[very thick, midarrow] (v2) -- (v4);
				\draw[very thick, midarrow] (v2) -- (v5);
				
			\end{tikzpicture}
			\ \\[10pt]
			\begin{tikzpicture}[scale=1,
				midarrow/.style={
					decoration={markings,
						mark=at position 0.6 with {\arrow{Stealth[length=3mm,width=2mm]}}},
					postaction={decorate}
				}
				]
				
				\node at (-1,4) { \textbf{Links:}};
				
				% ------- Link 1 -------
				\node at (-1,3) {1.};
				\begin{scope}[shift={(0.5,3)}]
					\node[circle,fill=black,inner sep=2pt] (a1) at (0,0) {};
					\node[circle,fill=black,inner sep=2pt] (a2) at (1.0,0) {};
					\draw[very thick,midarrow] (a1)--(a2);
				\end{scope}
				
				% ------- Link 2 -------
				\node at (-1,2) {2.};
				\begin{scope}[shift={(0.5,2)}]
					\node[circle,fill=black,inner sep=2pt] (b1) at (0,0) {};
					\node[circle,fill=black,inner sep=2pt] (b2) at (1.2,0) {};
					\node[circle,fill=black,inner sep=2pt] (b3) at (0.6,0.8) {};
					\draw[very thick,midarrow] (b1)--(b3);
					\draw[very thick,midarrow] (b3)--(b2);
					\draw[very thick,midarrow] (b2)--(b1);
				\end{scope}
				
				% ------- Link 3 -------
				\node at (-1,1) {3.};
				\begin{scope}[shift={(0.5,1)}]
					\node[circle,fill=black,inner sep=2pt] (c1) at (0,0) {};
					\node[circle,fill=black,inner sep=2pt] (c2) at (1.2,0) {};
					\node[circle,fill=black,inner sep=2pt] (c3) at (0.6,0.8) {};
					\draw[very thick,midarrow] (c1)--(c3);
					\draw[very thick,midarrow] (c3)--(c2);
					\draw[very thick,midarrow] (c1)--(c2);
				\end{scope}
				
				% ------- Link 4 -------
				\node at (-1,0) {4.};
				\begin{scope}[shift={(0.5,0)}]
					\node[circle,fill=black,inner sep=2pt] (d1) at (0,0) {};
					\node[circle,fill=black,inner sep=2pt] (d2) at (1.2,0) {};
					\node[circle,fill=black,inner sep=2pt] (d3) at (0.6,0.8) {};
					\node[circle,fill=black,inner sep=2pt] (d4) at (2.4,0) {};
					\draw[very thick,midarrow] (d3)--(d1);
					\draw[very thick,midarrow] (d3)--(d2);
					\draw[very thick,midarrow] (d2)--(d1);
					\draw[very thick,midarrow] (d2)--(d4);
				\end{scope}
				
				% ------- Link 5 -------
				\node at (-1,-1) {5.};
				\begin{scope}[shift={(0.5,-1)}]
					\node[circle,fill=black,inner sep=2pt] (e1) at (0,0) {};
					\node[circle,fill=black,inner sep=2pt] (e2) at (1.2,0) {};
					\node[circle,fill=black,inner sep=2pt] (e3) at (0.6,0.8) {};
					\node[circle,fill=black,inner sep=2pt] (e4) at (2.4,0) {};
					\draw[very thick,midarrow] (e1)--(e3);
					\draw[very thick,midarrow] (e3)--(e2);
					\draw[very thick,midarrow] (e1)--(e2);
					\draw[very thick,midarrow] (e2)--(e4);
				\end{scope}
				
			\end{tikzpicture}
			
			& 
			\begin{tikzpicture}[scale=1,
				midarrow/.style={
					decoration={markings,
						mark=at position 0.6 with {\arrow{Stealth[length=3mm,width=2mm]}}},
					postaction={decorate}
				}
				]
				
				\node at (-1,4) { \textbf{Links:}};
				
				% ------- Link 1 -------
				\node at (-1,3) {1.};
				\begin{scope}[shift={(0.5,3)}]
					\node[circle,fill=black,inner sep=2pt] (a1) at (0,0) {};
					\node[circle,fill=black,inner sep=2pt] (a2) at (1.0,0) {};
					\draw[very thick,midarrow] (a1)--(a2);
				\end{scope}
				
				% ------- Link 2 -------
				\node at (-1,2) {2.};
				\begin{scope}[shift={(0.5,2)}]
					\node[circle,fill=black,inner sep=2pt] (b1) at (0,0) {};
					\node[circle,fill=black,inner sep=2pt] (b2) at (1.2,0) {};
					\node[circle,fill=black,inner sep=2pt] (b3) at (0.6,0.8) {};
					\draw[very thick,midarrow] (b1)--(b3);
					\draw[very thick,midarrow] (b3)--(b2);
					\draw[very thick,midarrow] (b2)--(b1);
				\end{scope}
				
				% ------- Link 3 -------
				\node at (-1,1) {3.};
				\begin{scope}[shift={(0.5,1)}]
					\node[circle,fill=black,inner sep=2pt] (c1) at (0,0) {};
					\node[circle,fill=black,inner sep=2pt] (c2) at (1.2,0) {};
					\node[circle,fill=black,inner sep=2pt] (c3) at (0.6,0.8) {};
					\draw[very thick,midarrow] (c1)--(c3);
					\draw[very thick,midarrow] (c3)--(c2);
					\draw[very thick,midarrow] (c1)--(c2);
				\end{scope}
				
				% ------- Link 4 -------
				\node at (-1,0) {4.};
				\begin{scope}[shift={(0.5,0)}]
					\node[circle,fill=black,inner sep=2pt] (d1) at (0,0) {};
					\node[circle,fill=black,inner sep=2pt] (d2) at (1.2,0) {};
					\node[circle,fill=black,inner sep=2pt] (d3) at (0.6,0.8) {};
					\node[circle,fill=black,inner sep=2pt] (d4) at (2.4,0) {};
					\draw[very thick,midarrow] (d3)--(d1);
					\draw[very thick,midarrow] (d3)--(d2);
					\draw[very thick,midarrow] (d2)--(d1);
					\draw[very thick,midarrow] (d4)--(d2);
				\end{scope}
				
				% ------- Link 5 -------
				\node at (-1,-1) {5.};
				\begin{scope}[shift={(0.5,-1)}]
					\node[circle,fill=black,inner sep=2pt] (e1) at (0,0) {};
					\node[circle,fill=black,inner sep=2pt] (e2) at (1.2,0) {};
					\node[circle,fill=black,inner sep=2pt] (e3) at (0.6,0.8) {};
					\node[circle,fill=black,inner sep=2pt] (e4) at (2.4,0) {};
					\draw[very thick,midarrow] (e1)--(e3);
					\draw[very thick,midarrow] (e3)--(e2);
					\draw[very thick,midarrow] (e1)--(e2);
					\draw[very thick,midarrow] (e2)--(e4);
				\end{scope}
				
			\end{tikzpicture}
			
			\end{tabular}
		\caption{Two distinct link-irregular digraphs on 5 vertices.}
		\label{tournament5}
	\end{figure}
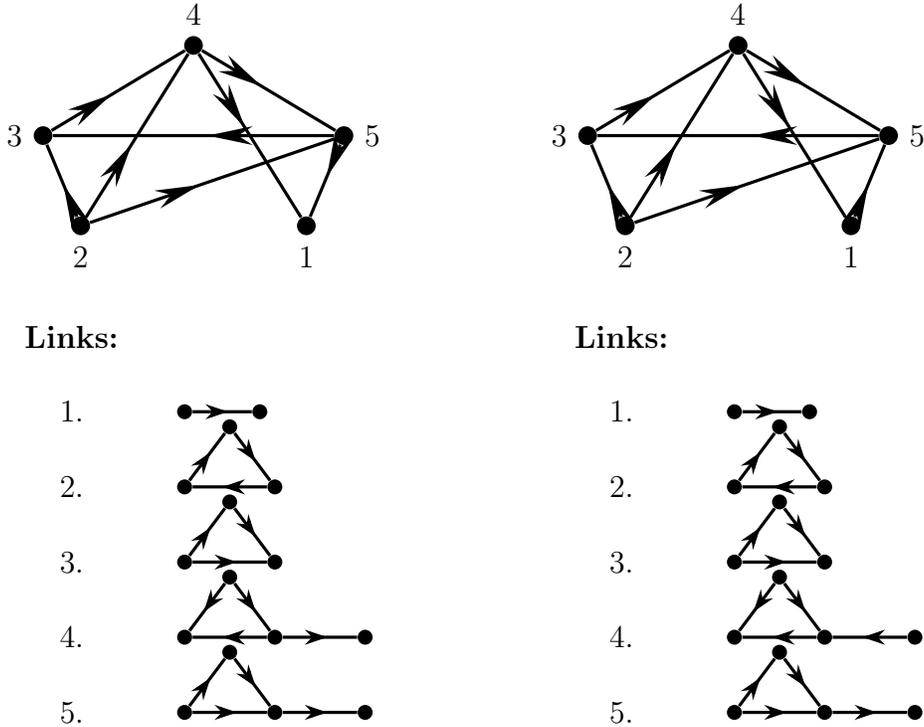
    
    Suppose that $n > 5$. Then by \cite{akbar_book}, there exists a link-irregular graph on $n$ vertices. Since any isomorphism of digraphs is an isomorphism of the underlying graphs and a link-irregular graph has no non-trivial automorphisms, no two orientations of a link-irregular graph are isomorphic. Since, by \cite{alexander}, any link-irregular graph on $n$ vertices has at least $2n-5 > 0$ edges, we have at least $2^{2n-5} > 1$ distinct link-irregular digraphs on $n$ vertices.
\end{proof}

In studying link-irregular digraphs, we explored the concept of a link-irregular tournament. In particular, from our explorations, we developed the following conjecture which has proven to be rather tricky to prove.

\begin{conj} \label{tournamentthm}
A link-irregular tournament exists on $n$ vertices if and only if $n\geq 6$.
\end{conj}

We have developed an explicit proof for the case $n \leq 8$ and a computer verification of the conjecture for the case $n\leq 100$. We present these now.

\begin{obs}
    For $n \leq 8$, Conjecture~\ref{tournamentthm} holds.
\end{obs}
\begin{proof}
    We consider the "only if", part of the statement first. By Theorem \ref{existencethm}, we only have to show that there does not exist a link-irregular tournament on $5$ vertices. This is the case. For if a link-irregular tournament on 5 vertices existed, the link of each vertex would be an orientation of $K_4$, but there are only four distinct orientations of $K_4$. Since there are $5$ vertices, some two vertices must have isomorphic links.\\
    Now we show that there does indeed exist a link-irregular tournament on $n$ vertices for $6\leq n \leq 8$. We claim the tournament on $6$ vertices in Figure \ref{tournamentfig} is a link-irregular tournament. \\

    We first note the in-degree sets of the link each vertex in this graph:
    \begin{center}
        \begin{tabular}{c    c}
          $L(1):  \{1,1,1,3,4\}$  \ \   &  \ \  $L(4):  \{0,2,2,3,3\}$ \\
          $L(2):  \{1,1,2,2,4\}$ \ \    &  \ \  $L(5):  \{1,1,2,3,3\}$  \\
          $L(3):  \{1,1,2,3,3\}$  \ \  & \ \  $L(6):  \{1,2,2,2,3\}$ \\
        \end{tabular}    
    \end{center}
    The only two vertices that have the same in-degree sequence are $L(3)$ and $L(5)$. We look closer at these two links. Each of these links has a vertex of in-degree $2$. We count the number of directed $3$-cycles that contain this vertex of degree $2$. In $L(3)$, we count three such cycles while we only observe one in $L(5)$. So these links are not isomorphic. We have proved that this tournament is link-irregular. Call this tournament $D_6$.\\
    
    To form the link-irregular tournament $D_7$, we add a vertex $w$ to $D_6$ and add edges from this new vertex to each vertex of $D_6$. It is immediate that if $L(u) \cong L(v)$ in $D_7$, then $L(u)\cong L(v)$ in $D_6$ as well. If $L(u)\cong L(w)$ for some vertex $u$ of $D_6$, then since $w$ is a vertex of out-degree $5$ in $L(u)$, there must be some vertex of out-degree $5$ in $L(w)$. But $L(w) = D_6$, and $D_6$ has no vertex of out-degree 5. Hence $D_7$ is link-irregular.\\
    
    To form the link-irregular tournament $D_8$, we add a vertex $z$ to $D_7$ and add edges from each vertex of $D_7$ to $z$. As before, $L(u) \ncong L(v)$ for $u,v$ vertices in $D_7$. If $L(u) \cong L(z)$ then $L(u)$ has a vertex of in-degree $6$, hence $L(z)=D_7$ does too. But $D_7$ has no vertex of in-degree $6$. Hence $D_8$ is link-irregular.
\end{proof}
\begin{figure}[H]
	\centering
			\begin{tikzpicture}[scale=1.5, baseline=(current bounding box.north),
			midarrow/.style={decoration={markings, mark=at position 0.5 with {\arrow{Stealth[length=3mm, width=2mm]}}}, postaction={decorate}}]
			
			% Define vertices in a circular layout
			\foreach \i in {1,...,6} {
				\node[circle, fill=black, inner sep=2.5pt, label={90+60*(1-\i)}:\i] (v\i) at ({90+60*(1-\i)}:1.5cm) {};
			}
			
			% Draw all directed edges with arrows at midpoint
			\draw[very thick, midarrow] (v1) to[bend right=15] (v6);
			\draw[very thick, midarrow] (v1) to[bend right=15] (v3);
			\draw[very thick, midarrow] (v1) to[bend right=15] (v4);
			\draw[very thick, midarrow] (v2) to[bend right=15] (v1);
			\draw[very thick, midarrow] (v3) to[bend right=15] (v2);
			\draw[very thick, midarrow] (v3) to[bend right=15] (v4);
			\draw[very thick, midarrow] (v3) to[bend right=15] (v6);
			\draw[very thick, midarrow] (v4) to[bend right=15] (v5);
			\draw[very thick, midarrow] (v4) to[bend right=15] (v6);
			\draw[very thick, midarrow] (v4) to[bend right=15] (v2);
			\draw[very thick, midarrow] (v5) to[bend right=15] (v1);
			\draw[very thick, midarrow] (v5) to[bend right=15] (v2);
			\draw[very thick, midarrow] (v5) to[bend right=15] (v3);
			\draw[very thick, midarrow] (v5) to[bend right=15] (v6);
			\draw[very thick, midarrow] (v6) to[bend right=15] (v2);
		\end{tikzpicture}

	\caption{A link-irregular tournament on 6 vertices, $D_6$. }
	\label{tournamentfig}
\end{figure}
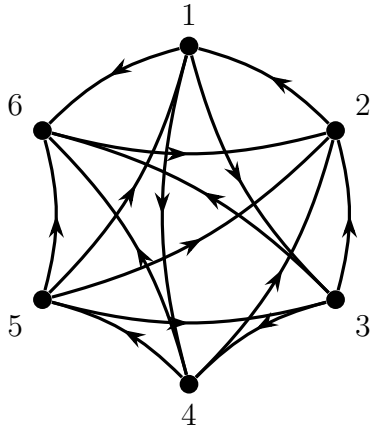

We had attempted to extend this construction to all values of $n$. However, if we were to repeat this as a two-step process, then we would end up having two isomorphic links in $D_9$. In particular, these would be the links of the vertex added to obtain $D_9$ and the vertex added to obtain $D_7$. When attempting to account for this and adjust our construction, it quickly became unruly and impossible to work with. Despite this flaw in our attempted construction, we have reason to believe the conjecture. \\

We developed a computer program to search for link-irregular tournaments, employing a multi-strategy approach including random generation, hill-climbing with edge flips, and seeded extension from known examples. Link graph isomorphism is detected using 
invariant signatures (vertex/edge counts, degree sequences, and directed 3-cycle counts) for efficient filtering, followed by full isomorphism testing via the VF2 algorithm~\cite{networkx}. Algorithm~\ref{alg:search} outlines our approach.

\begin{algorithm}[H]
	\caption{Computational Search for Link-Irregular Tournaments}
	\label{alg:search}
	\begin{algorithmic}[1]
		\Require Number of vertices $n$
		\Ensure Link-irregular tournament $T$ or failure report
		\State $T \gets$ \Call{RandomSearch}{$n$, max\_attempts $= 300$}
		\If{$T$ is link-irregular}
		\State \Return $T$
		\EndIf
		\State $T \gets$ \Call{HillClimbing}{$n$, steps $= 6000$, restarts $= 5$}
		\If{$T$ is link-irregular}
		\State \Return $T$
		\EndIf
		\State $T \gets$ \Call{SeededExtension}{$n$, attempts $= 50$}
		\State \Return $T$ or \textsc{null}
	\end{algorithmic}
\end{algorithm}

Using this algorithm, we computationally verified the following result. The upper bound $n = 100$ was chosen based on reasonable computational time; the algorithm 
can be extended to larger values with additional computing resources. Complete source code and detailed algorithmic specifications are available at 
\url{https://github.com/omidkhormali/Link-irregular-Digraphs}.

\begin{obs}
    For $n \leq 100$, Conjecture~\ref{tournamentthm} holds.
\end{obs}

% \textcolor{red}{discuss computer program (algorithm?) here and link repo. Note, may have to change 500 to some smaller value for the computer program to run}\\

Finally, in the following remark, we note a potential avenue of attack to prove this conjecture.

\begin{remark}
    In \cite{alexander2}, the authors of this paper introduced the concept of a link-irregular labeling. This idea can naturally be extended to digraphs by replacing graph isomorphism in all definitions with digraph isomorphism. In this paper, it is proven that there exists a link-irregular 2-labeling of $K_n$ if and only if $n \geq 6$. Any orientation of such a labeling is a link-irregular 2-labeled tournament.\\
    Suppose that we have a link-irregular 2-labeled tournament. We would like to show that if we change one edge label from its current label to the other possible label, we can change the directions of the edges in the tournament so that we still have a link-irregular 2-labeled tournament. This has proven rather tricky and is the sticking point of this idea for a proof. If, however, this can be proven, then the conjecture follows immediately. We can change the edge labels and adjust the edge orientations repeatedly until we obtain a monochromatic link-irregular digraph. But then this is a link-irregular tournament.
\end{remark}

Next, we will adapt a result from \cite{alexander} to obtain an analogous result on edge bounds and planarity for link-irregular digraphs.

\begin{theorem}\label{lbound}
    Let $f(n)$ denote the minimum number of edges in a link-irregular digraph on $n$ vertices. Then $f(n)=\Omega(n\sqrt{\log n})$.
\end{theorem}
\begin{proof}
    This proof follows almost exactly as the analogous proof from \cite{alexander}. An upper bound on the number of isomorphism classes of graphs on $h$ vertices is $2^{\binom{h}{2}}$. Since there are two possible orientations for each edge, we have an upper bound of $(2^{\binom{h}{2}})^2 = 2^{2 \binom{h}{2}}$ isomorphism classes of digraphs on $h$ vertices.
    Given a number $n$ of vertices, we may pick $k$ such that $2^{2 \binom{k}{2}} \leq n < 2^{2 \binom{k+1}{2}}$. For every value $d$ such that $1 \leq d < k$, a link-irregular graph on $n$ vertices has at most $2^{2\binom{d}{2}}$ vertices of degree $d$. Such a graph also has at least $n - \sum_{d=1}^{k-1} 2^{2\binom{d}{2}}$ vertices of degree at least $k$. We have
    \[
    \sum d(v) \geq \sum_{d = 1}^{k-1}d2^{2\binom{d}{2}} + k(n - \sum_{d=1}^{k-1} 2^{2\binom{d}{2}}) = kn - \sum_{d = 1}^{k-1}(k-d)2^{2\binom{d}{2}}.
    \]
    Since $2^{2\binom{k}{2}} \leq n < 2^{2\binom{k+1}{2}}$, $k^2-k \leq \log n < k^2+k $. Hence, as $n \rightarrow \infty$, $k \rightarrow \sqrt{\log n}$. In addition, we have
    \[\sum_{d = 1}^{k-1}(k-d)2^{2\binom{d}{2}} < k\sum_{d = 1}^{k-1}2^{2\binom{d}{2}} < k\sum_{i = 1}^{2\binom{k-1}{2}} 2^i = k(\frac{1-2^{2\binom{k-1}{2}}}{1-2}) = k(2^{2\binom{k-1}{2}} - 1).
    \]
    Since $n \geq 2^{2\binom{k}{2}}$, the term $nk$ dominates $k(2^{2\binom{k-1}{2}} - 1)$ and therefore dominates the term $\sum_{d = 1}^{k-1}(k-d)2^{2\binom{d}{2}}$. Combining this with the first inequality and the degree-sum formula, we obtain $f(n) = \Omega(\frac{n\sqrt{\log n}}{2}) = \Omega(n\sqrt{\log n})$.
\end{proof}
% We note that in the term $d=1$ in the sums, we assume the convention of defining $\binom{d}{2} = 0$ for $d < 2$.\\

% We observe that this implies that almost all link-irregular digraphs are nonplanar (that is, their underlying graph is nonplanar). For, by this result, there exists some natural number $N$ such that for $n \geq N$, any link irregular graph on $n$ vertices has more than $3n-6$ edges. We have this since $n\sqrt{\log n}$ grows at a faster rate than any linear function. By \cite{west}, any graph on $n$ vertices with more than $3n-6$ edges is nonplanar. Hence we have our

% \begin{theorem}
%     All but finitely many link-irregular digraphs are planar.
% \end{theorem}

We note that in the term $d=1$ in the sums, we assume the convention of defining $\binom{d}{2} = 0$ for $d < 2$.\\

We observe that this implies that almost all link-irregular digraphs are nonplanar (that is, their underlying graph is nonplanar). For, by this result, there exists some natural number $N$ such that for $n \geq N$, any link irregular digraph on $n$ vertices has more than $3n-6$ edges. We have this since $n\sqrt{\log n}$ grows at a faster rate than any linear function. By \cite{west}, any graph on $n$ vertices with more than $3n-6$ edges is nonplanar. Hence we have the following.

\begin{theorem}
    All but finitely many link-irregular digraphs are planar.
\end{theorem}

This result shows that link-irregular digraphs typically have dense underlying graphs, which naturally leads us to explore their relationship with other graph properties. In particular, we now investigate the connection between link-irregular orientable graphs and link-irregular labeling.\\

To proceed, we recall a result from \cite{alexander2}.

\begin{obs}\cite{alexander2} \label{labelingthm}
    Let $ G$ be a graph. Then $G$ has a link-irregular labeling if and only if given any distinct $ x, y \in V(G) $, either $ L(x) \not\cong L(y) $, or  $ E(L(x)) \neq E(L(y)) $.
\end{obs}

We are now ready to state our theorem. 

\begin{theorem}
    Suppose $G$ is a link-irregular orientable graph. Then $G$ is link-irregular labelable.
\end{theorem}

\begin{proof}
If $L(u) \not\cong L(v)$ for any distinct vertices $u, v \in V(G)$, then the result follows immediately since the underlying graph is already link-irregular. Suppose now that $L(u) \cong L(v)$ in $G$ for distinct vertices $u, v \in V(G)$. If $E(L(u)) = E(L(v))$, then we must also have $E(L_D(u)) = E(L_D(v))$ for any orientation $D$ of $G$. Consequently, $L_D(u) \cong L_D(v)$ for every orientation $D$, and so $G$ is not link-irregular orientable, contrary to our hypothesis. Therefore, we must have $E(L(u)) \neq E(L(v))$, and hence, by Observation~\ref{labelingthm}, $G$ is link-irregular labelable.
\end{proof}

Following this result, the natural question is whether the converse is true. It turns out that the converse is not true.

\begin{obs}
    The following graph is link-irregular 2-labelable, but not link-irregular orientable.
	
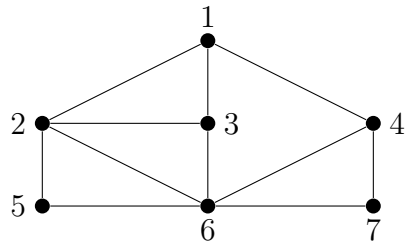
\begin{figure}[H]
		\centering
		\begin{tikzpicture}[scale=1.1,
			every node/.style={circle,fill=black,inner sep=2pt},
			textnode/.style={inner sep=0pt,fill=none}
			]
			
			% ---- Vertex positions ----
			\node (v1) at (0,2) {};
			\node (v2) at (-2,1) {};
			\node (v3) at (0,1) {};
			\node (v4) at (2,1) {};
			\node (v5) at (-2,0) {};
			\node (v6) at (0,0) {};
			\node (v7) at (2,0) {};
			
			% ---- Labels for vertices ----
			\node[textnode, above=4pt] at (v1) {1};
			\node[textnode, left=4pt]  at (v2) {2};
			\node[textnode, right=4pt] at (v3) {3};
			\node[textnode, right=4pt] at (v4) {4};
			\node[textnode, left=4pt]  at (v5) {5};
			\node[textnode, below=4pt] at (v6) {6};
			\node[textnode, below=4pt] at (v7) {7};
			
			% ---- Edges with R/B labels ----
			% R edges (straight)
			\draw (v1)--(v2) ;
			\draw (v1)--(v4) ;
			\draw (v1)--(v3) ;
			\draw (v3)--(v6) ;
			\draw (v2)--(v3) ;
			\draw (v2)--(v5) ;
			\draw (v2)--(v6) ;
			\draw (v5)--(v6) ;
			\draw (v4)--(v7) ;
			
			% B edges
			\draw (v4)--(v6) ;
			\draw (v6)--(v7) ;
			
		\end{tikzpicture}
		%\caption{A link-irregular $2$-labelable graph. The label R denotes a red edge, and B denotes a blue edge.}
		\label{graph}
	\end{figure}

\end{obs}
\begin{proof}
    To prove that this graph has a link-irregular 2-labeling, we provide one. 
	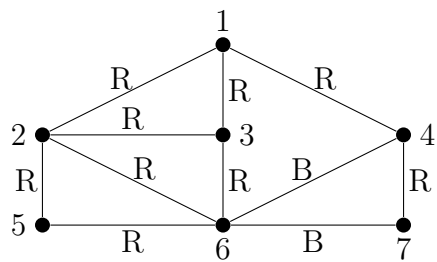
\begin{figure}[H]
		\centering
		\begin{tikzpicture}[scale=1.2,
			every node/.style={circle,fill=black,inner sep=2pt},
			textnode/.style={inner sep=0pt,fill=none}
			]
			
			% ---- Vertex positions ----
			\node (v1) at (0,2) {};
			\node (v2) at (-2,1) {};
			\node (v3) at (0,1) {};
			\node (v4) at (2,1) {};
			\node (v5) at (-2,0) {};
			\node (v6) at (0,0) {};
			\node (v7) at (2,0) {};
			
			% ---- Labels for vertices ----
			\node[textnode, above=4pt] at (v1) {1};
			\node[textnode, left=4pt]  at (v2) {2};
			\node[textnode, right=4pt] at (v3) {3};
			\node[textnode, right=4pt] at (v4) {4};
			\node[textnode, left=4pt]  at (v5) {5};
			\node[textnode, below=4pt] at (v6) {6};
			\node[textnode, below=4pt] at (v7) {7};
			
			% ---- Edges with R/B labels ----
			% R edges (straight)
			\draw (v1)--(v2) node[midway, above left, textnode]{R};
			\draw (v1)--(v4) node[midway, above right, textnode]{R};
			\draw (v1)--(v3) node[midway,  right, textnode]{R};
			\draw (v3)--(v6) node[midway,  right, textnode]{R};
			\draw (v2)--(v3) node[midway, above, textnode]{R};
			\draw (v2)--(v5) node[midway, left, textnode]{R};
			\draw (v2)--(v6) node[midway, above right, textnode]{R};
			\draw (v5)--(v6) node[midway, below, textnode]{R};
			\draw (v4)--(v7) node[midway, right, textnode]{R};
			
			% B edges
			\draw (v4)--(v6) node[midway, above left, textnode]{B};
			\draw (v6)--(v7) node[midway, below, textnode]{B};
			
		\end{tikzpicture}
		\caption{A link-irregular $2$-labelable graph. The label R denotes a red edge, and B denotes a blue edge.}
		\label{labeled}
	\end{figure}
    
    To see that this graph has no link-irregular orientation, we notice that two vertices have $K_2$ as a link. Since $K_2$ only has one orientation up to isomorphism, there is no possible link-irregular orientation of this graph.
\end{proof}

%%%%%%%%%%%%%%%%%%%%%%%%%%%%%%%%%%%%%
We now turn our attention to special classes of digraphs with regularity constraints. In particular, we investigate whether link-irregularity can coexist with strong connectivity and balanced degree properties. To formalize this discussion, we first recall the relevant definitions.

\begin{defi} \cite{west} [Strongly Connected Digraph]
		A directed graph (digraph) $D = (V, A)$ is said to be \emph{strongly connected} if for every pair of vertices $u, v \in V$, there exists a directed path from $u$ to $v$ and a directed path from $v$ to $u$.
	\end{defi}
	
	\begin{defi} \cite{west} [Eulerian Digraph]
		A digraph $D = (V, A)$ is called \emph{Eulerian} if it is strongly connected and for every vertex $v \in V$, the in-degree equals the out-degree; that is, $d^+(v) = d^-(v)$.
	\end{defi}

	The following observation follows from the preceding definitions.
    
	\begin{obs}
		For all integers $n \geq 3$ and $1 \leq k \leq n - 1$, there exist Eulerian digraphs on $n$ vertices with outdegree $k$ that are not link-irregular.
	\end{obs}
	
	\begin{proof}
		Fix $n \geq 3$ and let $k \in \{1, 2, \dots, n-1\}$. Define the digraph $D = (V, A)$ as follows:
		\[
		V = \{v_0, v_1, \dots, v_{n-1}\},
		\]
		\[
		A = \{ v_i \to v_{(i + j) \bmod n} \mid 1 \leq j \leq k,\ 0 \leq i \leq n - 1 \}.
		\]
		In this construction, each vertex has outdegree $k$, sending arcs to the next $k$ vertices cyclically. Also, each vertex also has indegree $k$, since it receives arcs from the $k$ preceding vertices modulo $n$. The digraph is strongly connected because the steps $\{1, 2, \dots, k\}$ generate the full group $\mathbb{Z}_n$ (especially when $\gcd(k, n) = 1$), and in general the structure ensures reachability via directed paths. Thus, $D$ is Eulerian. However, due to the rotational symmetry of the construction, all vertices have isomorphic directed links. Specifically, for each vertex $v_i$, the link $L_D(v_i)$ is induced by the same relative configuration of neighbors. Therefore, $D$ is not link-irregular.
	\end{proof}

% \vspace{1in}
% \begin{ques}
% 	A direction for further exploration is the role of outdegree regularity. Specifically:
	
% 	\begin{itemize}
% 		\item Can one prove that no link-irregular digraph exists in which all vertices have the same outdegree $k$?
% 		\item More generally, what can be said about the existence of link-irregular digraphs in which all outdegrees are at least (or at most) $k$?
% 	\end{itemize}
	
% 	Do such constraints on the outdegree sequence inherently limit the possibility of link-irregularity?
% \end{ques}

We now present several observations on the existence of link-irregular digraphs with specific properties. 

\begin{obs}
    There exists a link-irregular digraph on $6$ vertices in which all vertices have the same outdegree $2$.
\end{obs}
\begin{proof}
    Consider the digraph $D_6$ on vertex set $V = \{0, 1, 2, 3, 4, 5\}$ with edge set as shown in Figure~\ref{fig:d6_2out}. Each vertex has outdegree exactly $2$, making this a 2-out-regular digraph. Direct computation shows that all vertices have distinct link-degree sequences, confirming that $D_6$ is link-irregular.
\end{proof}

\begin{figure}[H]
		\centering
		\begin{tikzpicture}[scale=1.3, baseline=(current bounding box.north),
			midarrow/.style={decoration={markings, mark=at position 0.5 with {\arrow{Stealth[length=2.5mm, width=1.5mm]}}}, postaction={decorate}}]
			
			% Define vertices - custom layout to match the image
			% Vertex 3 at top
			\node[circle, fill=black!50, inner sep=3pt, label=above:3] (v3) at (0, 2) {};
			% Vertex 5 in upper middle
			\node[circle, fill=black!50, inner sep=3pt, label=left:5] (v5) at (-0.5, 1) {};
			% Vertex 2 on left
			\node[circle, fill=black!50, inner sep=3pt, label=left:2] (v2) at (-1.5, 0.3) {};
			% Vertex 4 on right
			\node[circle, fill=black!50, inner sep=3pt, label=right:4] (v4) at (1.5, 0.3) {};
			% Vertex 0 at bottom
			\node[circle, fill=black!50, inner sep=3pt, label=below:0] (v0) at (0, -1.2) {};
			% Vertex 1 on lower right
			\node[circle, fill=black!50, inner sep=3pt, label=right:1] (v1) at (1.2, -0.5) {};
			
			% Draw all directed edges with arrows at midpoint
			% Edges from vertex 0
			\draw[thick, midarrow] (v0) to[bend right=10] (v1);
			\draw[thick, midarrow] (v0) to[bend right=10] (v2);
			
			% Edges from vertex 1
			\draw[thick, midarrow] (v1) to[bend right=10] (v4);
			\draw[thick, midarrow] (v1) to[bend right=10] (v2);
			
			% Edges from vertex 2
			\draw[thick, midarrow] (v2) to[bend right=10] (v5);
			\draw[thick, midarrow] (v2) to[bend right=10] (v4);
			
			% Edges from vertex 3
			\draw[thick, midarrow] (v3) to[bend right=10] (v4);
			\draw[thick, midarrow] (v3) to[bend right=10] (v2);
			
			% Edges from vertex 4
			\draw[thick, midarrow] (v4) to[bend right=10] (v5);
			\draw[thick, midarrow] (v4) to[bend right=10] (v0);
			
			% Edges from vertex 5
			\draw[thick, midarrow] (v5) to[bend right=10] (v0);
			\draw[thick, midarrow] (v5) to[bend right=10] (v3);
			
		\end{tikzpicture}
		\caption{Link-irregular 2-out-regular digraph on 6 vertices.}
		\label{fig:d6_2out}
	\end{figure}
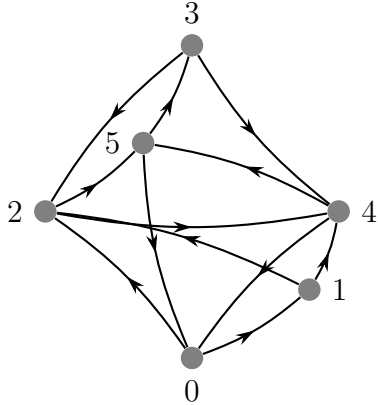

% \begin{prob} \label{outregularprob}
%     For which values of $n$ and $k$ does there exist a link-irregular digraph on $n$ vertices with constant out-degree $k$?
% \end{prob}

% \begin{ques}
%     Does there exist a link-irregular digraph in which all vertices have the same outdegree and indegree $k$?
% \end{ques}

We next consider link-irregular digraphs with both constant in-degree and constant out-degree.

\begin{obs}
    There exists a link-irregular digraph on $9$ vertices in which all vertices have the same outdegree and indegree $4$. Further, this graph is a tournament.
\end{obs}
\begin{proof}
    We construct an explicit example of such a digraph. Consider the tournament $D_9$ on vertex set $V = \{0, 1, 2, 3, 4, 5, 6, 7, 8\}$ with edge set as depicted in Figure~\ref{fig:k9_tournament}. By construction, each vertex has outdegree 4 and indegree 4, making this a regular tournament. To verify link-irregularity, we compute the link-outdegree sequence. A direct calculation shows that no two vertices share the same multiset of neighbor degrees, confirming that $D_9$ is link-irregular. Therefore, there exists a link-irregular regular tournament on 9 vertices.
\end{proof}

\begin{figure}[H]
		\centering
		\begin{tikzpicture}[scale=1.5, baseline=(current bounding box.north),
			midarrow/.style={decoration={markings, mark=at position 0.5 with {\arrow{Stealth[length=2.5mm, width=1.5mm]}}}, postaction={decorate}}]
			
			% Define vertices in a circular layout (9 vertices)
			\foreach \i in {0,...,8} {
				\node[circle, fill=black!50, inner sep=3pt, label={90+40*(0-\i)}:\i] (v\i) at ({90+40*(0-\i)}:2cm) {};
			}
			
			% Draw all directed edges with arrows at midpoint
			% Edges from vertex 0
			\draw[thick, midarrow] (v0) to[bend right=10] (v2);
			\draw[thick, midarrow] (v0) to[bend right=10] (v3);
			\draw[thick, midarrow] (v0) to[bend right=10] (v4);
			\draw[thick, midarrow] (v0) to[bend right=10] (v7);
			
			% Edges from vertex 1
			\draw[thick, midarrow] (v1) to[bend right=10] (v0);
			\draw[thick, midarrow] (v1) to[bend right=10] (v2);
			\draw[thick, midarrow] (v1) to[bend right=10] (v5);
			\draw[thick, midarrow] (v1) to[bend right=10] (v7);
			
			% Edges from vertex 2
			\draw[thick, midarrow] (v2) to[bend right=10] (v4);
			\draw[thick, midarrow] (v2) to[bend right=10] (v5);
			\draw[thick, midarrow] (v2) to[bend right=10] (v6);
			\draw[thick, midarrow] (v2) to[bend right=10] (v7);
			
			% Edges from vertex 3
			\draw[thick, midarrow] (v3) to[bend right=10] (v1);
			\draw[thick, midarrow] (v3) to[bend right=10] (v2);
			\draw[thick, midarrow] (v3) to[bend right=10] (v6);
			\draw[thick, midarrow] (v3) to[bend right=10] (v8);
			
			% Edges from vertex 4
			\draw[thick, midarrow] (v4) to[bend right=10] (v1);
			\draw[thick, midarrow] (v4) to[bend right=10] (v3);
			\draw[thick, midarrow] (v4) to[bend right=10] (v6);
			\draw[thick, midarrow] (v4) to[bend right=10] (v7);
			
			% Edges from vertex 5
			\draw[thick, midarrow] (v5) to[bend right=10] (v0);
			\draw[thick, midarrow] (v5) to[bend right=10] (v3);
			\draw[thick, midarrow] (v5) to[bend right=10] (v4);
			\draw[thick, midarrow] (v5) to[bend right=10] (v8);
			
			% Edges from vertex 6
			\draw[thick, midarrow] (v6) to[bend right=10] (v0);
			\draw[thick, midarrow] (v6) to[bend right=10] (v1);
			\draw[thick, midarrow] (v6) to[bend right=10] (v5);
			\draw[thick, midarrow] (v6) to[bend right=10] (v8);
			
			% Edges from vertex 7
			\draw[thick, midarrow] (v7) to[bend right=10] (v3);
			\draw[thick, midarrow] (v7) to[bend right=10] (v5);
			\draw[thick, midarrow] (v7) to[bend right=10] (v6);
			\draw[thick, midarrow] (v7) to[bend right=10] (v8);
			
			% Edges from vertex 8
			\draw[thick, midarrow] (v8) to[bend right=10] (v0);
			\draw[thick, midarrow] (v8) to[bend right=10] (v1);
			\draw[thick, midarrow] (v8) to[bend right=10] (v2);
			\draw[thick, midarrow] (v8) to[bend right=10] (v4);
			
		\end{tikzpicture}
		\caption{Link-irregular regular orientation of $K_9$.}
		\label{fig:k9_tournament}
	\end{figure}
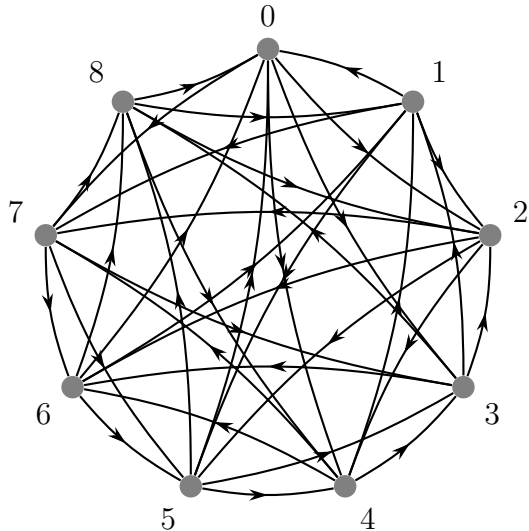

% \begin{prob} \label{tournamentprob}
%     For which values of $n$ and $k$ does there exist a link-irregular graph on $n$ vertices with constant in- and out-degrees both equal to $k$?
%     For which values of $n$ and $k$ does there exist a link-irregular tournament on $n$ vertices with constant out-degree $k$?
% \end{prob}

% We partially answer the Problem \ref{outregularprob} and the first part of Problem \ref{tournamentprob} by providing a necessary condition on $n$ and $k$ for them to satisfy either problem.
The edge sets for the digraphs in Figures \ref{fig:d6_2out} and \ref{fig:k9_tournament}  are available in \texttt{Digraph\_examples\_edgesets.txt} in the repository\footnote{\url{https://github.com/omidkhormali/Link-irregular-Digraphs}}.\\

We now establish lower bounds on the degree and outdegree that must exist in any link-irregular digraph, which provides the necessary conditions for 
link-irregularity.

\begin{theorem}
    Any link irregular digraph must have: 
    \begin{itemize}
        \item A vertex $u$ such that $d(u) \geq h - \sum_{d=1}^{h-1}\frac{h-d}{n}2^{2\binom{d}{2}}$, and
        \item A vertex $w$ such that $d^+(w) \geq \frac{1}{2}(h - \sum_{d=1}^{h-1}\frac{h-d}{n}2^{2\binom{d}{2}})$.
    \end{itemize}
    where $h = \lfloor \frac{1+\sqrt{1+4\log{n}}}{2}\rfloor$
\end{theorem}
\begin{proof}
    We saw in Theorem \ref{lbound} that 
    \[\sum d(v) \geq  hn - \sum_{d = 1}^{h-1}(h-d)2^{2\binom{d}{2}}\]
    where $h$ is the maximum integer satisfying $2^{2\binom{h}{2}} \leq n$, $n$ the number of vertices in our digraph. Taking the log of both sides of this inequality, we obtain $2\binom{h}{2} = h^2-h \leq \log{n}$ which is equivalent to $h^2-h-\log{n} \leq 0$. Setting the left-hand side of this inequality equal to zero, we obtain $h = \frac{1 \pm \sqrt{1 + 4\log{n}}}{2}$. Since we wish to maximize $h$, and $h$ must be an integer, we have 
    \[h = \lfloor \frac{1+\sqrt{1+4\log{n}}}{2}\rfloor \].
    Using the first inequality we took from Theorem \ref{lbound}, we divide through by $n$ to obtain 
    \[\frac{1}{n}\sum d(v) \geq  h - \sum_{d = 1}^{h-1}\frac{h-d}{n}2^{2\binom{d}{2}}\].
    By the pigeonhole principle, we have the first statement of our theorem.
    To obtain the second statement of our theorem, we observe that $\sum d^+(v) = \frac{1}{2}\sum d(v)$. Then, dividing through our earlier equation by two, we obtain
    \[\frac{1}{n}\sum d^+(v) = \frac{1}{2n}\sum d(v) \geq  \frac{1}{2}(h - \sum_{d = 1}^{h-1}\frac{h-d}{n}2^{2\binom{d}{2}}).\]
    As before, by the pigeonhole principle, we have the second statement of our theorem.
\end{proof}

An immediate consequence of this result is the following corollary.
\ \\
\begin{cor}
    Any digraph in which:
    \begin{itemize}
        \item All vertices have a degree less than $h - \sum_{d=1}^{h-1}\frac{h-d}{n}2^{2\binom{d}{2}}$, or
        \item All vertices have an outdegree less than $\frac{1}{2}(h - \sum_{d=1}^{h-1}\frac{h-d}{n}2^{2\binom{d}{2}})$
    \end{itemize}
    is not link-irregular.
\end{cor}

Then in any link-irregular digraph with constant outdegree, all vertices must have an outdegree greater than or equal to $\frac{1}{2}(h - \sum_{d=1}^{h-1}\frac{h-d}{n}2^{2\binom{d}{2}})$. That is, 
%in Problems \ref{outregularprob} and \ref{tournamentprob}, we must have
a link-irregular digraph on $n$ vertices with constant outdegree $k$ must satisfy $k \geq \frac{1}{2}(h - \sum_{d=1}^{h-1}\frac{h-d}{n}2^{2\binom{d}{2}})$ where $h = \lfloor \frac{1+\sqrt{1+4\log{n}}}{2}\rfloor$.

\section*{Acknowledgments}
This research was supported by the UExplore Undergraduate Research Program at the University of Evansville.

\end{document}